\documentclass{amsart}

\usepackage[dvipdfmx]{graphicx}
\usepackage{color}

\usepackage{float}
\usepackage{array,booktabs}
\usepackage{amssymb}
\usepackage{subcaption}
\usepackage{txfonts}
\usepackage{amsmath}
\usepackage{subcaption}
\usepackage{bm}
\usepackage{mathrsfs}
\usepackage{overpic}
\usepackage{multimedia}

\newtheorem{theorem}{Theorem}[section]

\theoremstyle{definition}
\newtheorem{definition}[theorem]{Definition}
\newtheorem{example}[theorem]{Example}

\newtheorem{cor}[theorem]{Corollary}
\theoremstyle{remark}

\newcommand{\ack}{\noindent \textbf{Acknowledgements. }}
\newcommand{\org}{\noindent \textbf{Organization. }}

\DeclareMathOperator{\Ehr}{Ehr}
\DeclareMathOperator{\ehr}{\varepsilon}
\DeclareMathOperator{\T}{Top}

\newcommand{\conv}{\operatorname{Conv}}

\numberwithin{equation}{section}

\title[New recursion formula for the interior polynomial]{New recursion formula for the interior polynomial\\ based on non-expanding sets}
\author{Keiju Kato}
\date{\today}
\email{kato.k.at@m.titech.ac.jp}
\address{Department of Mathematics, Institute of Science Tokyo, Oh-okayama 2-12-1, Meguro-ku, Tokyo 152-8551, Japan}

\begin{document}
\maketitle

\begin{abstract}
The interior polynomial was originally defined for hypergraphs and later shown to coincide with the Ehrhart polynomial of the root polytope of an associated bipartite graph. In previous work, we derived an alternating cycle recursion formula for the interior polynomial. Here, we introduce a new, more transparent recursion formula based on the structure of non-expanding sets. This formula offers a clearer combinatorial interpretation of the interior polynomial and its connection to polyhedral geometry.
\end{abstract}

\section{Introduction}
In this paper, we investigate properties of the interior polynomial, which is an invariant of bipartite graphs. Initially, the interior polynomial was defined by K\'alm\'an as an invariant of hypergraphs \cite{K}. Here a hypergraph $\mathscr{H}=(V,W)$ has a vertex set $V$ and a hyperedge set $W$, where $W$ is a multiset of non-empty subsets of $V$. Originally, the interior polynomial was defined as the polynomial invariant associated with hypergraphs, but as shown by the main result of \cite{KP}, it can also be regarded as the polynomial invariant of the corresponding bipartite graph $H$ with color classes $V$ and $W$. The author extended the interior polynomial to signed bipartite graphs, that is, bipartite graphs $G$ with a sign $E \to \{ +1 , -1 \}$, where $E$ is the set of edges in the bipartite graph $G=(V\sqcup W, E)$ \cite{kato}. The signed interior polynomial $I^+_G$ is constructed as an alternating sum of the interior polynomials of the bipartite graphs obtained from $G$ by deleting some negative edges and disregarding the sign.

The interior polynomial is related to a part of the HOMFLY polynomial. The HOMFLY polynomial \cite{homfly} is a two-variable invariant of oriented links in $S^3$ defined by the skein relation
\[
v^{-1} P_{\includegraphics[width=0.35cm]{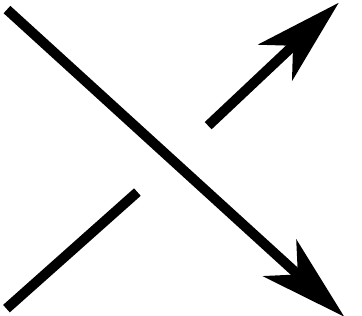}}(v,z)-
v      P_{\includegraphics[width=0.35cm]{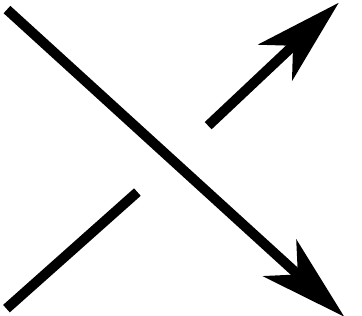}}(v,z)=
z      P_{\includegraphics[width=0.35cm]{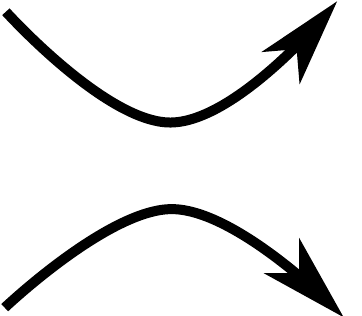}}(v,z),
\]
and the initial condition $P_{\mbox{\footnotesize unknot}}(v,z)=1$. Morton \cite{morton} showed that for any oriented link diagram $D$, the maximal $z$-exponent in the HOMFLY polynomial of an oriented link diagram $D$ is less than or equal to $c(D)-s(D)+1$, where $c(D)$ is the crossing number of $D$ and $s(D)$ is the number of its Seifert circles. We call the coefficient of $z^{c(D)-s(D)+1}$ (which is a polynomial in $v$) the top of the HOMFLY polynomial, denoted by $\T_D(v)$. When $G=(V\sqcup W,E_+\sqcup E_-)$ is a signed plane bipartite graph, the sequence of coefficients of $I^+_G(x)$ agrees with that of $\T_{L_G}(v)$, where $L_G$ is the special link diagram with Seifert graph $G$ \cite{kato}. More precisely,
\[
\T_{L_G}(v)=v^{|E_{+}|-|E_{-}|-(|V|+|W|)+1}I^{+}_{G}(v^2 ).
\]
This correspondence follows from its special case when $G$ is a positive graph, which in turn is established in two steps. First, the interior polynomial of $G$ is equivalent to the Ehrhart polynomial of the root polytope of $G$ \cite{KP}. The latter can be regarded as an h-vector \cite{KP} and coincides with $\T_{L_G}$ \cite{KM}.

Originally, the interior polynomial was defined as a generalization of $T_G(1/x,1)$, where $T_G(x,y)$ is the Tutte polynomial of a graph $G$. The Tutte polynomial has a deleting-contracting formula: for a graph $G$ and an edge $e$ which is neither a bridge nor a loop, $T_G = T_{G-e}+T_{G/e}$, where the graph $G-e$ is obtained from $G$ by deleting $e$ and $G/e$ is obtained from $G$ by contracting $e$. Using this formula, the Tutte polynomial can be computed recursively. In \cite{K}, we obtained a recursion formula for the interior polynomial based on alternating cycles (see Theorem \ref{thm:altrec}).

To obtain a new, more transparent recursion formula, we exploit the concept of a non-expanding set. A non-expanding set $S\subseteq V$ is one whose neighborhood is nonempty and no larger than itself, i.e.,
\[
0<|N_G(S)|\leq |S| .
\]
We also require that $S$ contains no isolated vertices of $G$. This property allows us to decompose the graph into induced subgraphs by deleting subsets of $S$, and, as a consequence, express the interior polynomial as an alternating sum over these subgraphs. The main result of this paper is the following.
\begin{theorem}\label{main}
Let $G=(V\sqcup W, E)$ be a bipartite graph with vertex set $V\sqcup W$ separated by color and let $S\subseteq V$ be a non-expanding set of $G$. Then, we have
\[
\sum_{J\subseteq S}(-1)^{|J|}I_{G- J}(x)=0,
\]
where $G-J$ is the bipartite graph obtained from $G$ by deleting all vertices in $J$ and all incident edges.
\end{theorem}

In addition, it is known that evaluating the interior polynomial at \(x=1\) yields the volume of the associated root polytope. In~\cite{LP}, the authors prove a vanishing result for a similar alternating sum of the volumes of root polytopes for certain subgraphs, which corresponds precisely to the \(x=1\) specialization of Theorem~\ref{main}. Hence, Theorem~\ref{main} extends their result from the volume case to a broader framework encompassing the entire interior polynomial.

\org In Section \ref{sec:int}, we review the necessary definitions and properties of the interior polynomial, including its interpretation via the Ehrhart series of root polytopes. In Section \ref{sec:specialproof}, we establish a special case of our main theorem by proving the alternating sum identity for non‐expanding sets under the bijection condition. In Section \ref{sec:proofmain}, we extend this argument to non-expanding sets without the bijection condition using induction and Hall’s Theorem. Finally, in Section \ref{sec:computation}, we illustrate the recursion formula with several examples, including the computation of the interior polynomial for complete bipartite graphs.

\ack I would like to express my sincere gratitude to Professor Tam\'as K\'alm\'an, my doctoral advisor, for his many comments and insightful feedback, which have been instrumental in the development of this paper. I also thank Kouki Sato for his careful reading of the proof and valuable suggestions.

\section{Preliminaries}\label{sec:int}
In \cite{KP,kato}, it was shown that the interior polynomial of a bipartite graph $G$ coincides with the Ehrhart polynomial of its root polytope. In this section, we review the relevant definitions and properties, beginning with the definition of the root polytope.

\begin{definition}
Let $G=(V\sqcup W, E)$ be a bipartite graph. For $v \in V$ and $w \in W$, let {\bf v} and {\bf w} denote the corresponding standard generators (standard basis vectors) of $\mathbb{R}^V \oplus \mathbb{R}^W$. Define the root polytope of $G$ by
\[
Q_G=\conv \{\, {\bf v} + {\bf w} \mid vw \mbox{ is an edge of }G \,\}.
\]
\end{definition}

It is known that if $G$ is connected, then the root polytope $Q_G$ has dimension $\dim Q_G = |V| + |W| - 2$, as shown in \cite{P}. Throughout this paper, we set $d = |V| + |W| - 2$.

\begin{definition}
Let $G=(V\sqcup W,E)$ be a bipartite graph and $Q_G$ be the root polytope of $G$. For any positive integer $s$, the Ehrhart polynomial is defined by
\[
\ehr_{Q_G}(s) = |(s\cdot Q_G) \cap (\mathbb{Z}^{V}\oplus\mathbb{Z}^{W})|.
\]
\end{definition}

In general, for any polytope $P$, the function $\varepsilon_P(s)$ defined analogously is not necessarily a polynomial. However, if $P$ is a convex polytope with integer vertices, $\varepsilon_P(s)$ becomes a polynomial. In particular, $\varepsilon_{Q_G}(s)$ is a polynomial.

\begin{definition}
Let $G$ be a bipartite graph, and let $\varepsilon_{Q_G}(s)$ be the Ehrhart polynomial of the root polytope $Q_G$. The Ehrhart series is defined by
\[
\Ehr_{Q_G}(x)=1+\sum_{s\in \mathbb{N}}\varepsilon_{Q_G}(s)x^s.
\]
\end{definition}

Notice that, for a (bipartite) graph with no edges, we have $\Ehr_{Q_G}(x)=1$. It was shown that the Ehrhart series of the root polytope $Q_G$ is equivalent to the interior polynomial of the bipartite graph $G$.

\begin{theorem}\label{lem:serint}
Let $G=(V\sqcup W,E)$ be a connected bipartite graph and $I_G(x)$ be the interior polynomial of $G$. Then, the following holds:
\[
\frac{I_G(x)}{(1-x)^{|V|+|W|-1}}=\Ehr_{Q_G}(x).
\]
\end{theorem}

Theorem \ref{lem:serint} is implicit in \cite{KP}. The author extended this result to any (unsigned but possibly disconnected) bipartite graph. In this paper, we use the theorem in \cite{kato} as the definition of the interior polynomial for an arbitrary bipartite graph.

\begin{definition}[\cite{kato}]\label{thm:disserint}
Let $G=(V\sqcup W,E)$ be a bipartite graph. The interior polynomial for any bipartite graph is defined as $I_G(x)$ satisfying the following equality.
\[
\frac{I_G(x)}{(1-x)^{|V|+|W|-1}}=\Ehr_{Q_G}(x).
\]
\end{definition}

If $G_1$ and $G_2$ are bipartite graphs and $G_1\cup G_2$ is the disjoint union of $G_1$ and $G_2$, then we remark that $I_{G_1\cup G_2}(x)=(1-x)I_{G_1}I_{G_2}$. Moreover, this definition of the interior polynomial enabled us to derive a recursion formula for it.

\begin{theorem}[\cite{kato}]\label{thm:altrec}
If the distinct edges $e_1, e_{n+1},e_2, e_{n+2},\cdots,e_n,e_{2n}$ form a cycle in a bipartite graph $G$, then we have
\[
I_G(x)=\sum_{\emptyset\neq S\subseteq \{e_1,e_2,\cdots,e_n\}}(-1)^{|S|-1}I_{G\setminus S}(x).
\]

\begin{figure}[H]
\centering
\includegraphics[width=3cm]{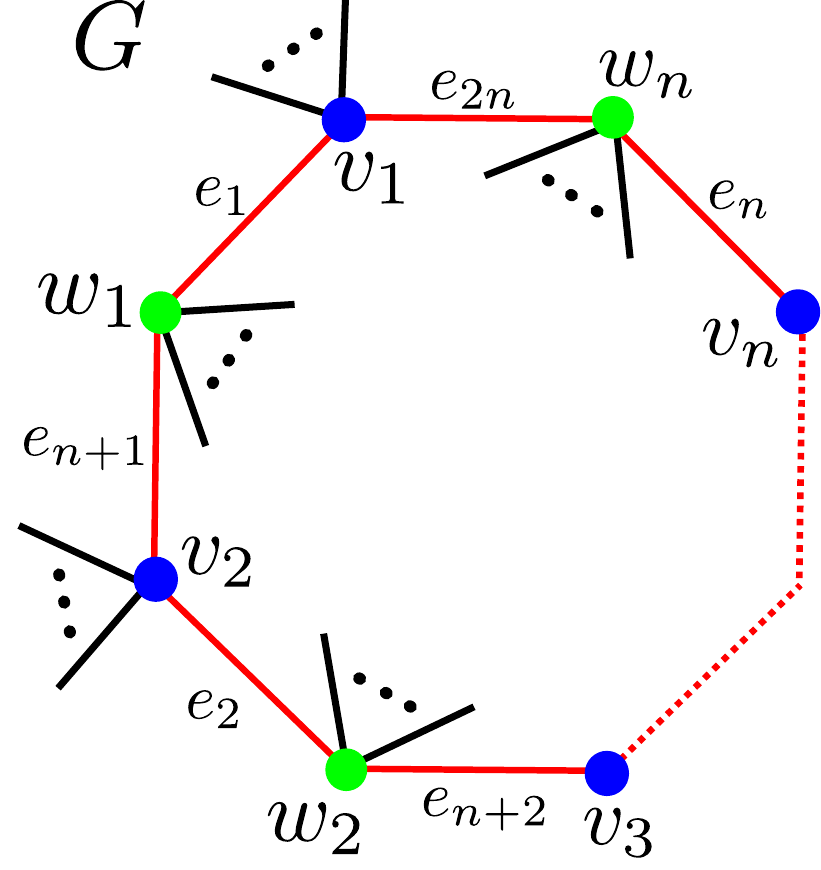}
\caption{Alternating cycle.}
\label{altcyc}
\end{figure}

\end{theorem}

One of the goals in this paper is to obtain a different recursion formula from Theorem \ref{thm:altrec}.

\section{Special case of Theorem \ref{main}}\label{sec:specialproof}
In this section, we prove a special case of Theorem \ref{main}. First, we precisely define what we mean by a non-expanding set.

\begin{definition}
For a vertex $ v\in V $, we define
\[
N_G(v) = \{ w \in W \mid w \text{ is adjacent to } v \}.
\]
For $ S\subseteq V $, we define
\[
N_G(S) = \bigcup_{v\in S}N_G(v).
\]
\end{definition}

\begin{definition}
We say a vertex set $S\subseteq V$ is a non-expanding set if 
\[
|S|\ge |N_G(S)|.
\]
\end{definition}

We can associate to $S$ a function $f \colon S \to N_G(S)$ by choosing, for each $v \in S$, a neighbor $f(v) \in N_G(S)$. If $f$ is a bijection, then it forms a perfect matching between $S$ and $N_G(S)$.

For any $J \subseteq V$, the bipartite graph $G-J$ is defined from $G$ by deleting all vertices in $J$ and all incident edges. The following theorem plays a key role in proving Theorem~\ref{main}.

\begin{theorem}\label{thm:special}
Let $G=(V\sqcup W,E)$ be a bipartite graph with the vertex set $V\sqcup W$ separated by color. If we have $|S|=|N_G(S)|$ and there exists a bijection $f:S\to N_G(S)$, then we have
\[
\sum_{J\subseteq S}(-1)^{|J|}I_{G- J}(x)=0,
\]
where $G-J$ is the bipartite graph obtained from $G$ by deleting all vertices in $J$ and all incident edges.
\end{theorem}

Before outlining the proof, we first present a concrete example to illustrate the key idea.

\begin{example}\label{computation}
Let $G$ be the bipartite graph shown in Figure \ref{fig:72}. If $S=\{v_0, v_1\}$, then Theorem \ref{thm:special} holds, as shown in Table~\ref{tab:1}.

\begin{figure}[H]
\centering
\includegraphics[width=2cm]{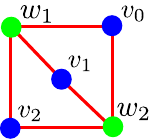}
\caption{A complete bipartite graph $K_{23}$.}
\label{fig:72}
\end{figure}

\begin{table}[htbp]
\centering
\caption{A computation of main theorem.}\label{tab:1}
\begin{tabular}{lllllllll}
& &  & & & $(-1)^{|J|}I_{G - J}$ \\
\hline 
\\ 

&\hspace{30pt}\includegraphics[width=0.75cm]{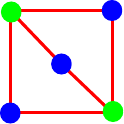}
&                                       &
&         
&\raisebox{+8pt}{\hspace{3pt}$1+2x$}   
&\raisebox{+8pt}{$\times 1$}\\

&\hspace{30pt}$G-\emptyset$  
&                                     &
&         
&  
&\\

& \includegraphics[width=0.75cm]{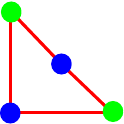}
& \includegraphics[width=0.75cm]{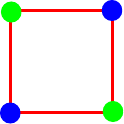}& 
&\raisebox{+8pt}{$-$\hspace{-10pt}}&\raisebox{+8pt}{$(1+1x)$}
           & \raisebox{+8pt}{$\times 2$}\\  
           
&\hspace{-5pt}$G-\{v_0\}$& \hspace{-5pt}$G-\{v_1\}$
                                    
&         
&  
&\\
           
&\hspace{30pt} \includegraphics[width=0.75cm]{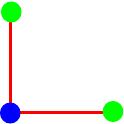}
& & 
& &\raisebox{+8pt}{$1$}
           & \raisebox{+8pt}{$\times 1$}\\  

&\hspace{20pt}$G-\{v_0,v_1\}\hspace{-10pt}$  
&                                     &
&         
&  
&\\

\hline \\
 & &\hspace{-50pt}$\displaystyle\sum_{J\subseteq S}(-1)^{|J|}I_{G-J}(x)$&$=$& &\hspace{3pt}$0$
\end{tabular}
\end{table}

\begin{figure}[htbp]
\begin{center}
\begin{tabular}{cc}
\hspace{-55pt}
\begin{minipage}{0.65\hsize}
\begin{center}
\includegraphics[height=3.5cm]{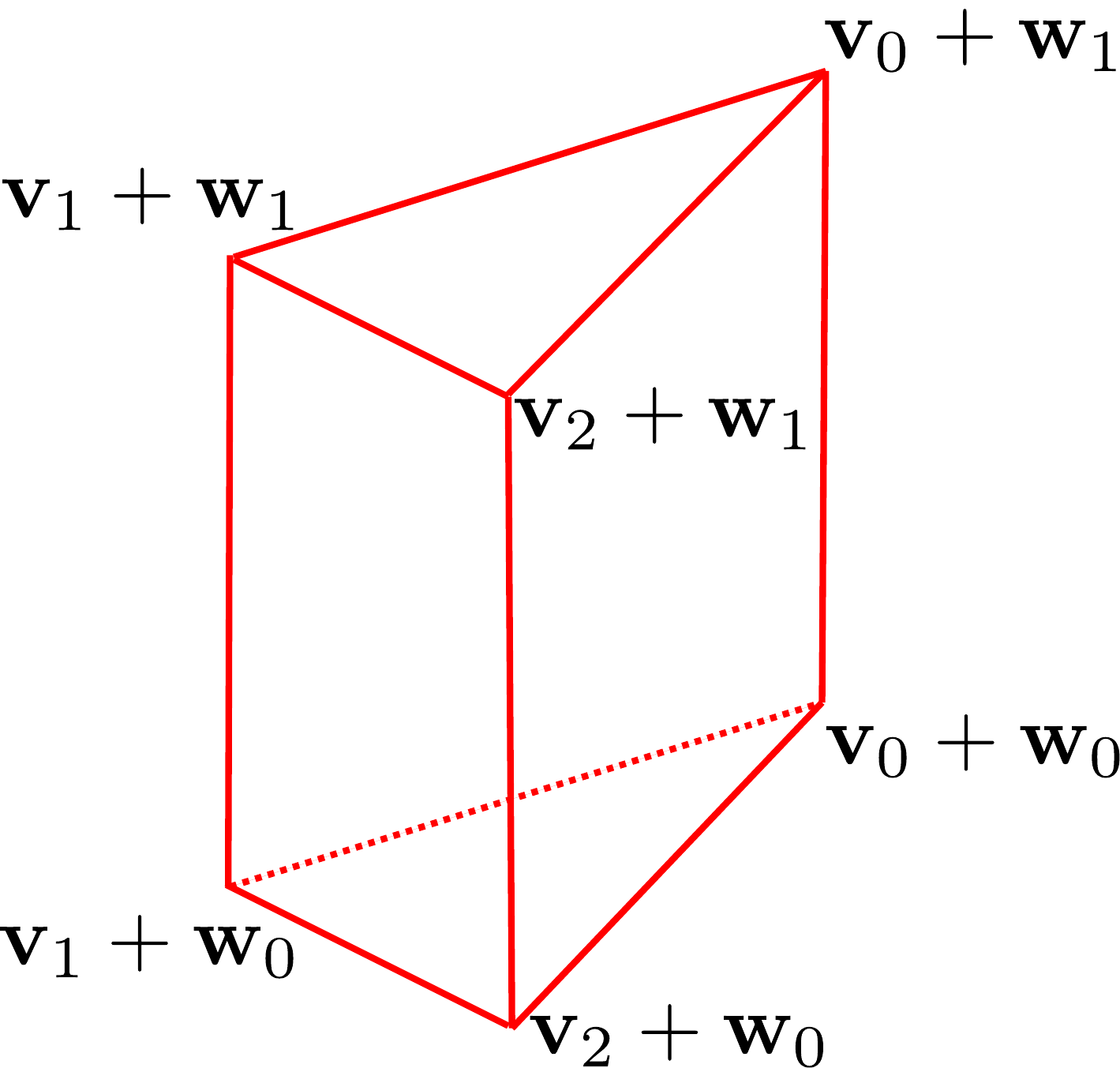}
\caption{The root polytope $Q_{K_{23}}$.}
\label{fig:P}
\end{center}
\end{minipage}
\hspace{-35pt}
\begin{minipage}{0.65\hsize}
\begin{center}
\includegraphics[height=3.5cm]{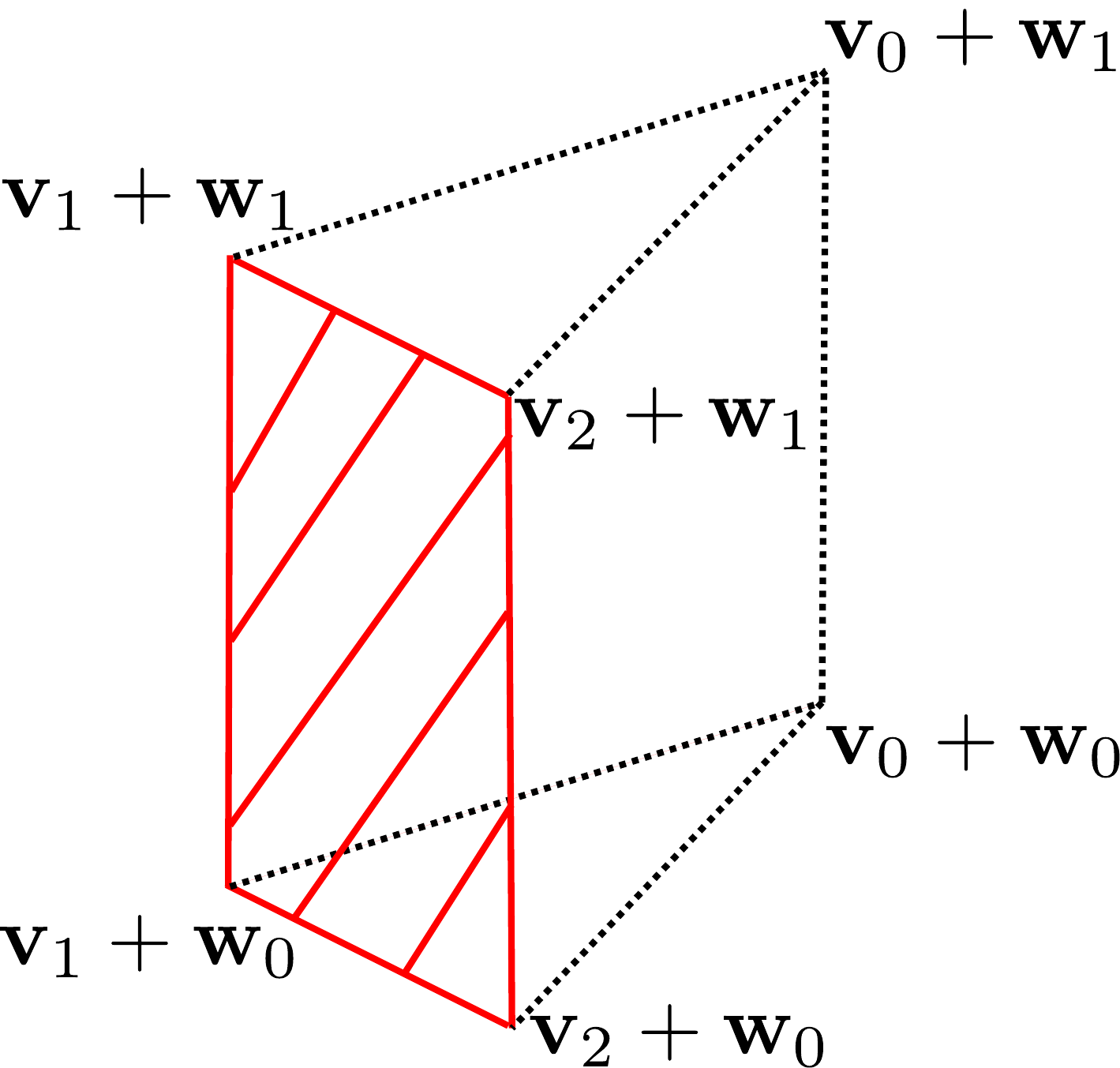}
\caption{The root polytope $Q_{K_{23}-\{v_0\}}$.}
\label{fig:P1}
\end{center}
\end{minipage}
\end{tabular}
\end{center}
\begin{center}
\begin{tabular}{cc}
\hspace{-35pt}
\begin{minipage}{0.65\hsize}
\begin{center}
\includegraphics[height=3.5cm]{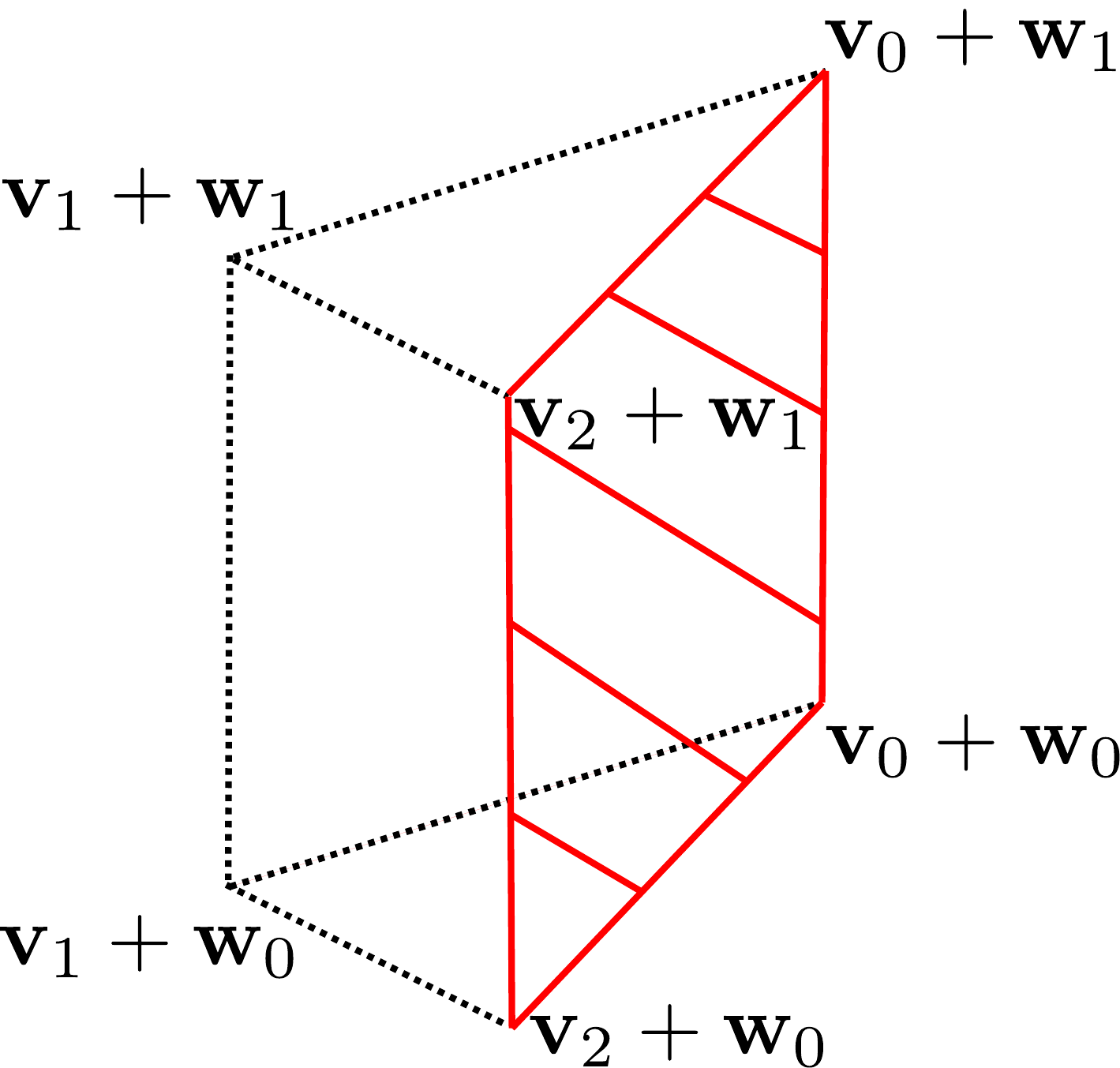}
\caption{The root polytope $Q_{K_{23}-\{v_1\}}$.}
\label{fig:P2}
\end{center}
\end{minipage}
\hspace{-25pt}
\begin{minipage}{0.65\hsize}
\begin{center}
\includegraphics[height=3.5cm]{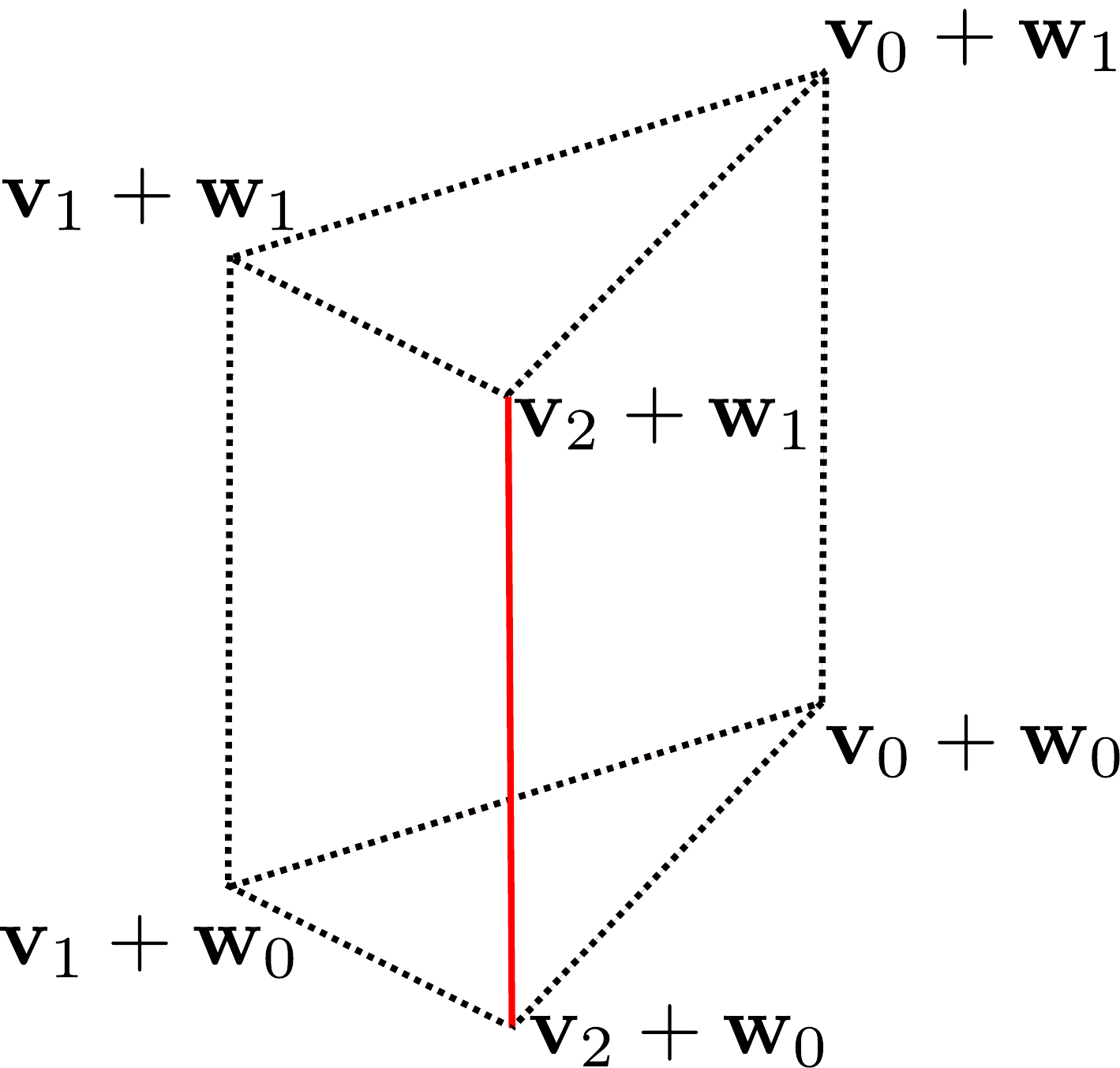}
\caption{The root polytope $Q_{K_{23}-\{v_0,v_1\}}$.}
\label{fig:P3}
\end{center}
\end{minipage}
\end{tabular}
\end{center}

\end{figure}

We now analyze this example through its Ehrhart series. Specifically, we compute the Ehrhart series of the root polytopes associated with the graphs in Table~\ref{tab:1}; see also Figures~\ref{fig:P} -- \ref{fig:P3}. We obtain
\[
\begin{split}
\frac{1+2x}{(1-x)^4} &= \Ehr_{Q_{K_{23}}}(x), \quad
\frac{1+x}{(1-x)^3} = \Ehr_{Q_{K_{23}-\{v_0\}}}(x), \\[1mm]
\frac{1+x}{(1-x)^3} &= \Ehr_{Q_{K_{23}-\{v_1\}}}(x), \quad
\text{and} \quad \frac{1}{(1-x)^2} = \Ehr_{Q_{K_{23}-\{v_0,v_1\}}}(x).
\end{split}
\]

Using the relation 
\[
\sum_{J\subseteq S}(-1)^{|J|}I_{G-J}(x)=0
\]
in this example, we obtain the following identity:
\[
(1-x)^2\Ehr_{Q_{K_{23}}}-(1-x)\Ehr_{Q_{K_{23}-\{v_0\}}}-(1-x)\Ehr_{Q_{K_{23}-\{v_1\}}}+\Ehr_{Q_{K_{23}-\{v_0,v_1\}}}=0.
\]
Considering the coefficient of $x^{k+2}$, we obtain
\begin{eqnarray*}
            \ehr_{Q_{K_{23}}}(k+2)&-2\ehr_{Q_{K_{23}}}(k+1) &+\ehr_{Q_{K_{23}}}(k)\\
   -\ehr_{Q_{K_{23}-\{v_0\}}}(k+2)&+\ehr_{Q_{K_{23}-\{v_0\}}}(k+1)& \\
   -\ehr_{Q_{K_{23}-\{v_1\}}}(k+2)&+\ehr_{Q_{K_{23}-\{v_1\}}}(k+1)& \\
+\ehr_{Q_{K_{23}-\{v_0,v_1\}}}(k+2).&
\end{eqnarray*}
We consider polytopes of the same size together and write
\[
\begin{split}
\widetilde{Q}_{01} &= Q_{K_{23}} - \Bigl( Q_{K_{23}-\{v_0\}} \cup Q_{K_{23}-\{v_1\}} \Bigr), \quad
\widetilde{Q}_0 = Q_{K_{23}} - Q_{K_{23}-\{v_0\}},\\[1mm]
\widetilde{Q}_1 &= Q_{K_{23}} - Q_{K_{23}-\{v_1\}}, \quad \text{and} \quad
\widetilde{Q}_\emptyset = Q_{K_{23}}.
\end{split}
\]
These polytopes are illustrated in Figures~\ref{fig:W}--\ref{fig:W3}. For instance, when considering the lattice point count in the $(k+2)$-dilation, it may initially seem necessary to compute the contributions from the four polytopes
\[
Q_{K_{23}},\quad Q_{K_{23}-\{v_0\}},\quad Q_{K_{23}-\{v_1\}},\quad \text{and} \quad Q_{K_{23}-\{v_0,v_1\}}.
\]
However, as a result of a useful cancellation, it suffices to count the lattice points in the single region $\widetilde{Q}_{01}$. Equivalently, the identity to be established can be written as
\[
\ehr_{\widetilde{Q}_{01}}(k+2)-\ehr_{\widetilde{Q}_0}(k+1)-\ehr_{\widetilde{Q}_1}(k+1)+\ehr_{\widetilde{Q}_\emptyset}(k)=0.
\]

\begin{figure}[htbp]
\begin{center}
\begin{tabular}{cc}
\hspace{-55pt}
\begin{minipage}{0.7\hsize}
\begin{center}
\includegraphics[height=3.5cm]{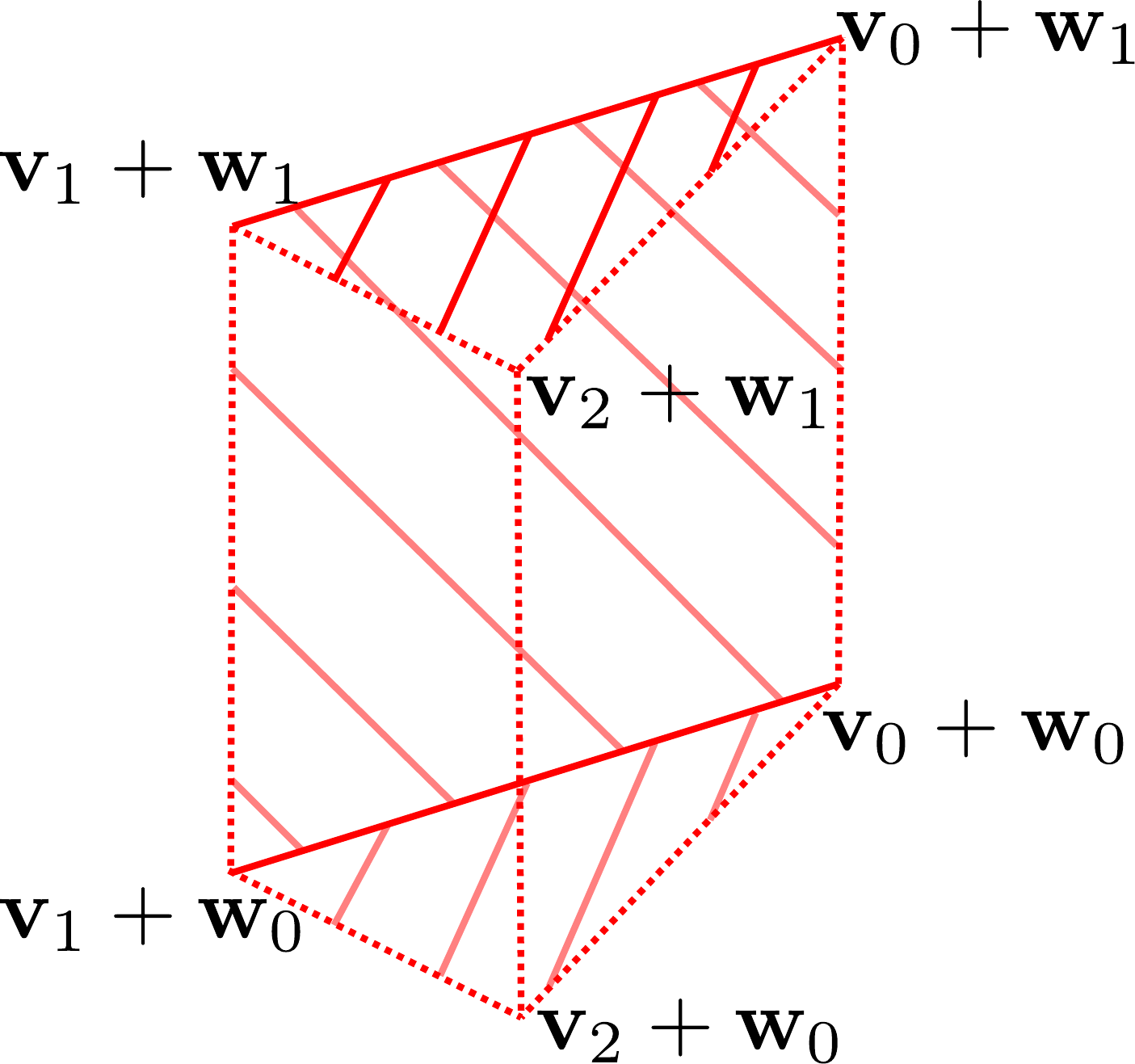}
\hspace{-30pt}\caption{The polytope $Q_{K_{23}}-(Q_{K_{23}-\{v_0\}}\cup Q_{K_{23}-\{v_1\}})$.}
\label{fig:W}
\end{center}
\end{minipage}
\hspace{-40pt}
\begin{minipage}{0.7\hsize}
\begin{center}
\includegraphics[height=3.5cm]{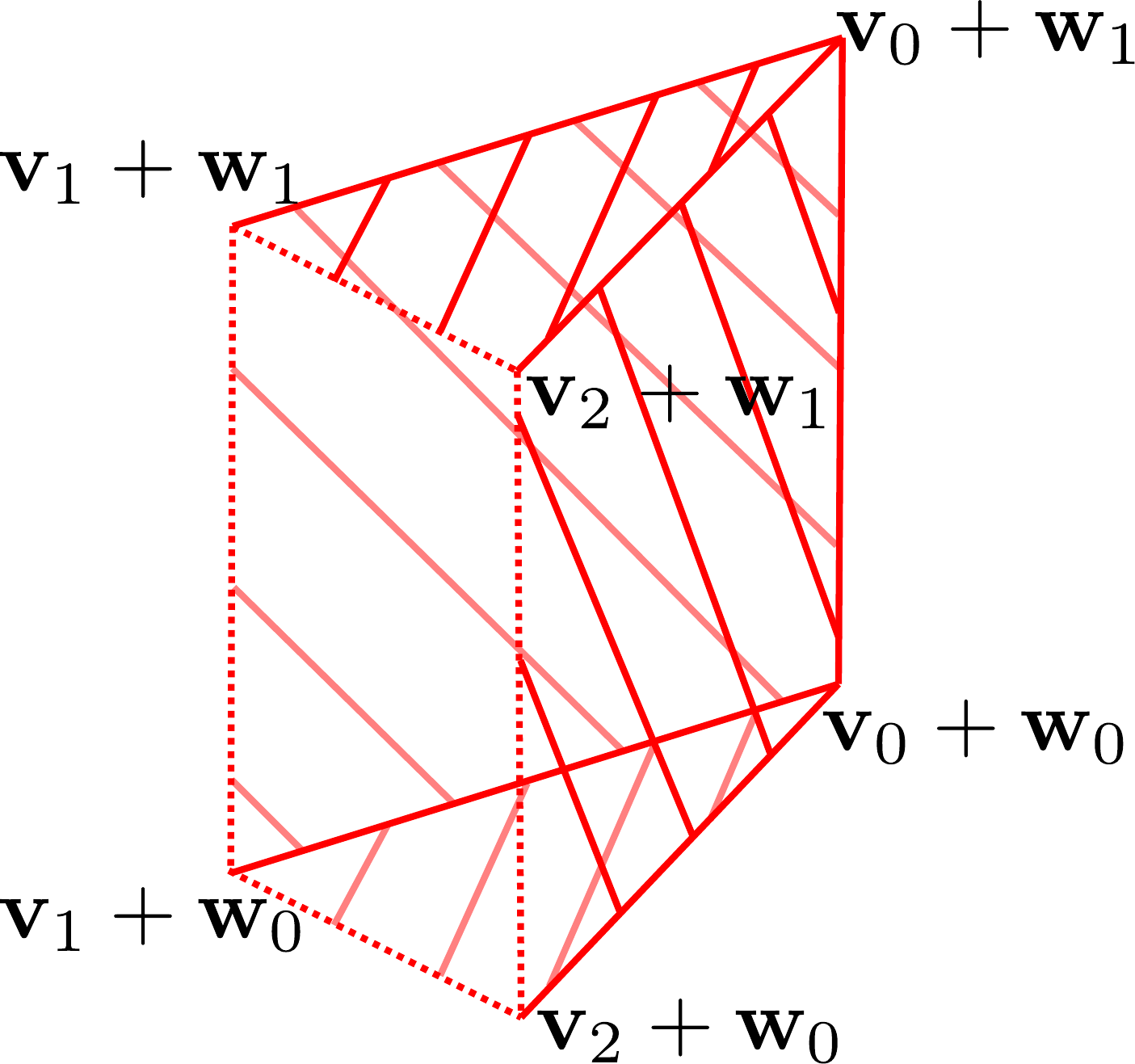}
\caption{The root polytope $Q_{K_{23}}-Q_{K_{23}-\{v_0\}}$.}
\label{fig:W0}
\end{center}
\end{minipage}
\end{tabular}
\end{center}
\begin{center}
\begin{tabular}{cc}
\hspace{-65pt}
\begin{minipage}{0.7\hsize}
\begin{center}
\includegraphics[height=3.5cm]{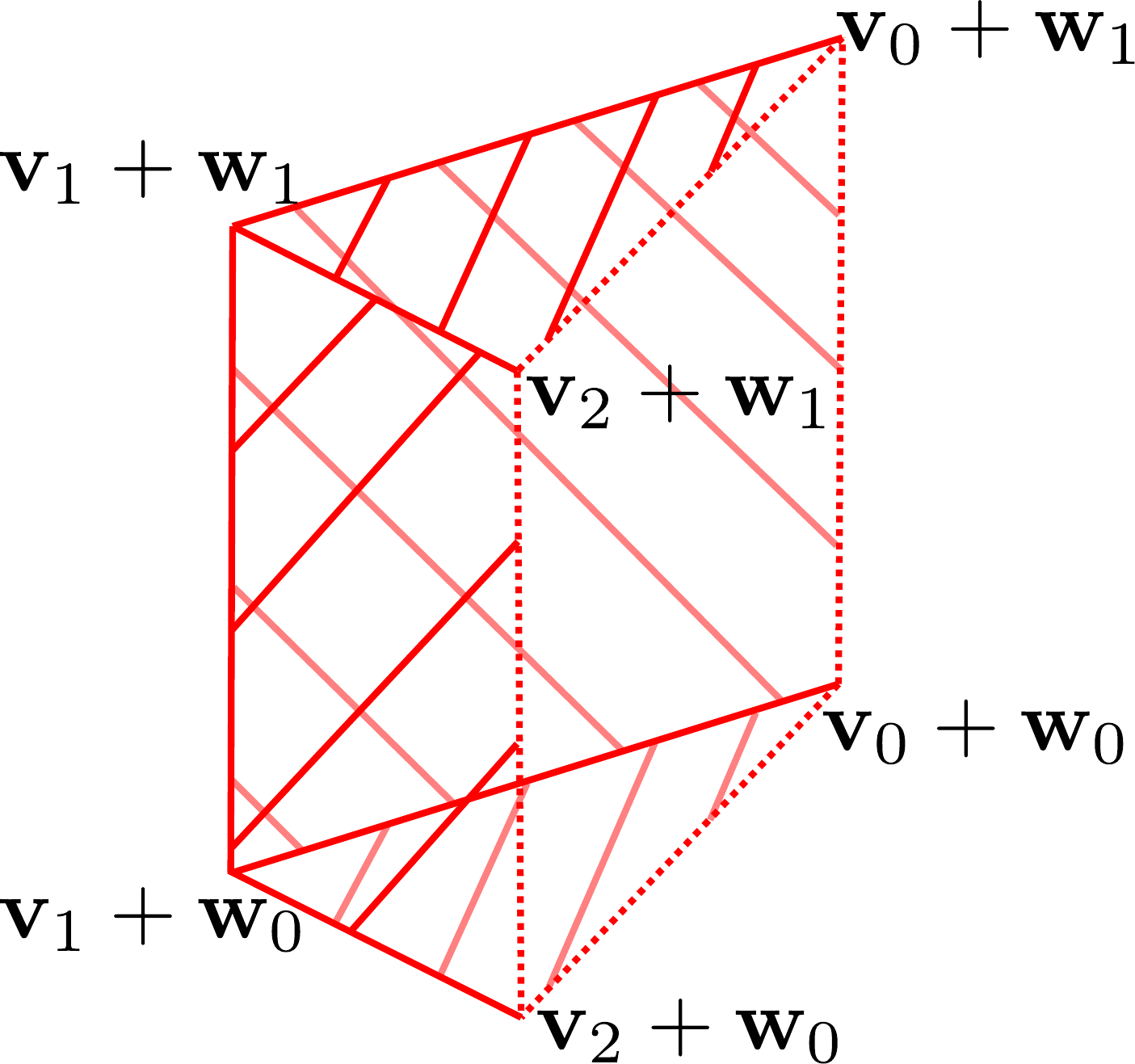}
\caption{The root polytope $Q_{K_{23}}-Q_{K_{23}-\{v_1\}}$.}
\label{fig:W1}
\end{center}
\end{minipage}
\hspace{-40pt}
\begin{minipage}{0.7\hsize}
\begin{center}
\includegraphics[height=3.5cm]{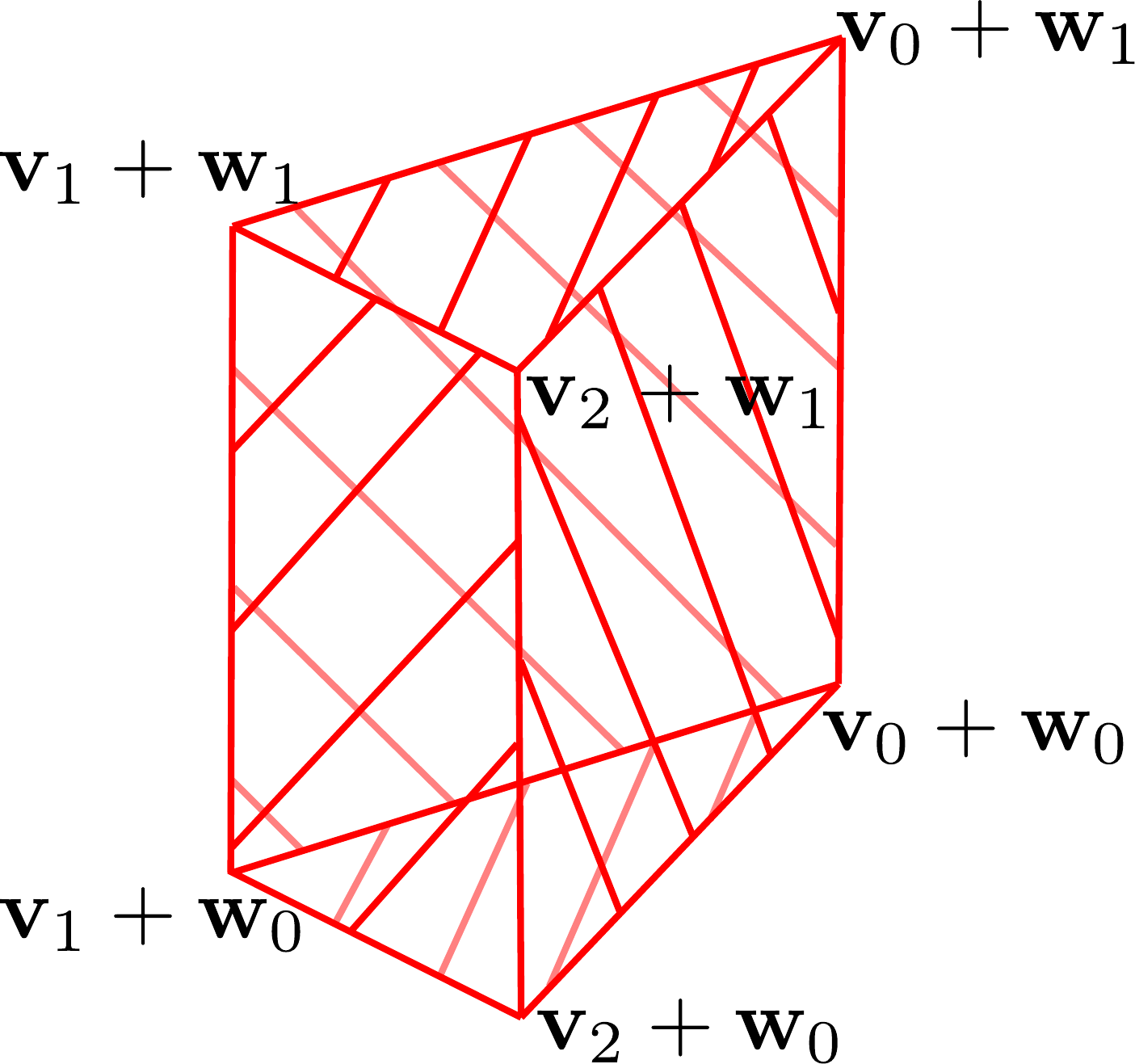}
\caption{The root polytope $Q_{K_{23}}$.}
\label{fig:W3}
\end{center}
\end{minipage}
\end{tabular}
\end{center}

\end{figure}

This identity admits a natural geometric interpretation. Here, $\ehr_{\widetilde{Q}_{01}}(k+2)$ counts the number of lattice points in the $(k+2)$-dilation of the region $\widetilde{Q}_{01}$. Meanwhile, we shift the other regions inside the dilated set
\[
\widetilde{Q}_{01}^+ = (k+2)\cdot\widetilde{Q}_{01}
\]
as follows:
\[
\widetilde{Q}_0^+ = (k+1)\cdot\widetilde{Q}_0 + \mathbf{v}_1 + \mathbf{e}_1,\quad
\widetilde{Q}_1^+ = (k+1)\cdot\widetilde{Q}_1 + \mathbf{v}_0 + \mathbf{e}_0,
\]
and
\[
\widetilde{Q}_\emptyset^+ = k\cdot\widetilde{Q}_\emptyset + \mathbf{v}_0 + \mathbf{e}_0 + \mathbf{v}_1 + \mathbf{e}_1.
\]
By construction, these shifted regions overlap in such a way that the alternating sum of their indicator functions cancels out:
\[
\bigl[\widetilde{Q}_{01}^+\bigr]-\bigl[\widetilde{Q}_0^+\bigr]-\bigl[\widetilde{Q}_1^+\bigr]+\bigl[\widetilde{Q}_\emptyset^+\bigr]=0.
\]
Considering the integer points in each region after suitable dilations and shifts, we see that the inclusion-exclusion sum of indicator functions cancels perfectly. Consequently, the alternating sum of their Ehrhart counts also vanishes, illustrating that this example satisfies Theorem \ref{thm:special}.

\end{example}

Theorem \ref{thm:special} is proved by considering indicator functions of lattice points in certain polytopes, as illustrated in Example~\ref{computation}. We establish the necessary notation.

Let $G = (V \sqcup W, E)$ be a bipartite graph, and let $S = \{v_1, \dots, v_n\} \subseteq V$ be a non-expanding set with $|S| = |N_G(S)|$. Denote by $Q_G$ the root polytope of $G$. For each $v_i \in V$, write $F_{v_i} = Q_{G - \{v_i\}}$. Since $G - \{v_i\}$ is a subgraph of $G$, we have $F_{v_i} \subseteq Q_G$. For any subset $U\subseteq S$, define the trimmed polytope
\[
  \widetilde{Q}_U 
  \;=\; 
  Q_G \,\setminus\, \bigcup_{v\in U}F_v.
\]
This polytope is obtained by removing all facets corresponding to the vertices in $U$. In particular, $\widetilde{Q}_\emptyset = Q_G$. We aim to show that the corresponding Ehrhart polynomials satisfy the alternating sum identity
\[
\sum_{U\subseteq S} (-1)^{|S\setminus U|}\,\ehr_{\widetilde{Q}_U}(k+|U|) = 0.
\]
We assume there exists a bijection $f : S \to N_G(S)$, and for any subset $U\subseteq S$, we write
\[
  \mathbf{v}_{S\setminus U} = \sum_{v_i \in S\setminus U} \mathbf{v}_i,
  \quad
  \mathbf{w}_{f(S\setminus U)} 
  = \sum_{w_i \in f(S\setminus U)} \mathbf{w}_i,
\]
where $\mathbf{v}_i$ and $\mathbf{w}_i$ are the standard generators of $\mathbb{R}^V \oplus \mathbb{R}^W$. Using this notation, we define the dilated and shifted polytope
\[
  \widetilde{Q}^+_U \;=\; 
  (k + |U|)\,\widetilde{Q}_U 
  \;+\;
  \mathbf{v}_{S\setminus U}
  \;+\;
  \mathbf{w}_{f(S\setminus U)}.
\]
In what follows, we show that the indicator functions of these polytopes sum to zero by alternating signs.

\begin{proof}[Proof of Theorem \ref{thm:special}]
We begin by noting that if $U_1 \subset U_2$ for subsets $U_1, U_2 \subseteq S$, then the trimmed polytopes satisfy the containment  
\[
\widetilde{Q}_{U_2} \subset \widetilde{Q}_{U_1}.
\]
This follows from the definition  
\[
\widetilde{Q}_U = Q_G \setminus \bigcup_{v \in U} F_v,
\]
which shows that removing more facets increases the number of excluded regions, resulting in a smaller polytope. On the other hand, for each $U \subseteq S$, we define the shifted and dilated polytope
\[
\widetilde{Q}_U^+ = (k+|U|)\,\widetilde{Q}_U + \mathbf{v}_{S\setminus U} + \mathbf{w}_{f(S\setminus U)},
\]
where the dilation factor $k+|U|$ increases as $U$ grows. Although $\widetilde{Q}_{U_2}$ is contained in $\widetilde{Q}_{U_1}$, the larger dilation factor applied to $\widetilde{Q}_{U_2}$ ensures that, after shifting, it expands beyond $\widetilde{Q}_{U_1}^+$. 

Consequently, we obtain the reverse inclusion
\[
\widetilde{Q}_{U_1}^+ \subset \widetilde{Q}_{U_2}^+.
\]

This phenomenon induces a nested structure among the polytopes $\{\widetilde{Q}_U^+\}_{U\subseteq S}$ that parallels the Boolean lattice of subsets of $S$. In such a Boolean-lattice framework, alternating sums over this nested family result in total cancellation. Concretely, we have the identity
\[
\sum_{U \subseteq S}
(-1)^{|S\setminus U|}
\Bigl[\,\widetilde{Q}_U^{+}\Bigr]
= 0.
\]
Summing over all lattice points in these polytopes gives
\[
\sum_{U\subseteq S}(-1)^{|S\setminus U|}\,\ehr_{\widetilde{Q}_U}(k+|U|)
= 0.
\]

Next, we relate this to the interior polynomial. Recall from Definition~\ref{thm:disserint} that if we set $d=|V|+|E|-2$, then for any $U\subseteq S$
\[
\Ehr_{Q_{G-U}}(x)
= \frac{I_{G-U}(x)}{(1-x)^{\,d+1-|U|}},
\]
which implies that
\[
I_{G-U}(x)
= (1-x)^{d+1-|U|}\,\Ehr_{Q_{G-U}}(x).
\]
Thus,
\begin{eqnarray*}
\Biggl(\sum_{U\subseteq S}(-1)^{|U|}I_{G-U}(x)\Biggr)\,\times\,\frac{1}{(1-x)^{d+1}}
&=& \sum_{U\subseteq S}(x-1)^{|U|}\,\Ehr_{Q_{G-U}}(x)\\[1mm]
&=& \cdots\;+\;\Biggl(\sum_{U\subseteq S}(-1)^{|S\setminus U|}\,\ehr_{\widetilde{Q}_U}(k+|U|)\Biggr)x^{k+|U|}+\cdots\\[1mm]
&=& 0.
\end{eqnarray*}
Therefore,
\[
\sum_{U\subseteq S}(-1)^{|U|} I_{G-U}(x)
= 0,
\]
which completes the proof.

\end{proof}

\section{Proof of main theorem}\label{sec:proofmain}

In this section, we extend Theorem \ref{thm:special} to prove the main theorem. To proceed with the proof, we first state Hall’s Theorem, which will play a crucial role in our argument.

\begin{theorem}[Hall's Theorem\cite{Hall}]
Let $G=(V\sqcup W, E)$ be a bipartite graph. There exists a perfect matching $V$ and $W$ if and only if we have $|X|\le |N_G(X)|$ for any $X\subseteq V$.
\end{theorem}

We now proceed to prove the main theorem by building upon the previous results. Specifically, we first assume that  $|S|=|N_G(S)|$ and show that this assumption is sufficient to complete the proof. The key steps involve induction on $ |S| - |N_G(S)| $ and utilizing Hall’s Theorem in the context of non-expanding sets.

\begin{proof}[Proof of Theorem \ref{main}]
Firstly, we show that under the assumption $|S| = |N_G(S)|$, it suffices to consider the case where $f: S \to N_G(S)$ is a bijection. By the contraposition of Hall's Theorem, if does not exist a bijective function $f:S\to N_G(S)$ (i.e., there is no perfect matching between $S$ and $N_G(S)$), we have there exists $X\subseteq S$ such that $|X|>|N_G(X)|$. Therefore, we have $X\ne S$, and $X$ is a non-expanding set. We have
\[
\sum_{J\subseteq S}(-1)^{|J|}I_{G-J}(x)
=\sum_{J_2\subseteq S\setminus X}(-1)^{|J_2|}\sum_{J_1\subseteq X}(-1)^{|J_1|}I_{(G-J_2)-J_1}(x).
\]
Obviously, we have $N_G(X)=N_{G-J_2}(X)$. So, the set $X$ is non-expanding set of $G-J_2$. Therefore we get
\[
\sum_{J\subseteq S}(-1)^{|J|}I_{G-J}(x)=0.
\]

Secondly, we prove by induction on $n = |S|$ that it suffices to treat the case $|S| = |N_G(S)|$. When $|S|=1$, write $S=\{v\}$.  Since $N_G(S)\neq\emptyset$ by definition, we have $|N_G(S)|=1$ and hence $|S|=|N_G(S)|$. Theorem~\ref{thm:special} thus applies immediately. Assume $|S|=k>1$.  By the induction hypothesis, the identity holds for every non-expanding set whose cardinality is strictly less than $k$. If $|S| - |N_G(S)| = 0$, then the claim follows from the reduction to the bijective case above. Otherwise $|S|-|N_G(S)|>0$, so pick any $v\in S$.  Then
\[
|S\setminus\{v\}|=k-1,
\qquad
|N_G(S\setminus\{v\})|\le|N_G(S)|\le k-1,
\]
hence $S\setminus\{v\}$ is a non‐expanding set both in $G$ and in $G-\{v\}$. Therefore, we have
\begin{eqnarray*}
\sum_{J\subseteq S}(-1)^{|J|}I_{G-J}(x)
&=&\sum_{J\subseteq S \atop v\in J}(-1)^{|J|}I_{G-J}(x)+\sum_{J\subseteq S\atop v\notin J}(-1)^{|J|}I_{G-J}(x)\\
&=&-\sum_{J_1\subseteq S \setminus \{v\}}(-1)^{|J_1|}I_{(G-\{v\})-J_1}(x)+\sum_{J_2\subseteq S\setminus \{v\}}(-1)^{|J_2|}I_{G-J_2}(x)\\
&=&-0+0\\
&=&0,
\end{eqnarray*}
which completes the proof.

\end{proof}

\section{Computation of the interior polynomial}\label{sec:computation}

In this section, we explain the recursion formula by presenting several concrete examples. Given a bipartite graph with color classes $V$ and $W$, we consider the color class with more vertices as the non-expanding set. By applying the main theorem, we obtain the following corollary.

\begin{cor}\label{rec}
Let $G=(V\sqcup W, E)$ be a bipartite graph with vertex set $V\sqcup W$ separated by color. If $|V| \geq |W|$, we have
\[
I_G(x)=\sum_{\emptyset \ne J\subseteq V}(-1)^{|J|-1}I_{G- J}(x),
\]
where $G-J$ is the bipartite graph obtained from $G$ by deleting all vertices in $J$ and all incident edges.
\end{cor}

At first, consider a bipartite graph consisting of isolated vertices (i.e., with no edges).As stated in \cite{kato}, if $G$ consists of $c(G)$ isolated vertices, then
\[
I_G(x) = (1-x)^{c(G)-1}.
\]
This is also consistent with the polyhedral definition given in Definition \ref{thm:disserint}.

\begin{example}
Let $H_{23}$ be the bipartite graph shown in Figure \ref{h23}. By applying Corollary \ref{rec} with either color class as the non-expanding set, the interior polynomial of $H_{23}$ can be computed inductively. We note that the computation follows a similar process regardless of which color class is chosen as the non-expanding set.

\begin{figure}[H]
\centering
\includegraphics[width=2cm]{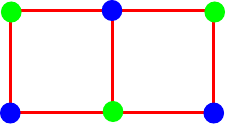}
\caption{a grid graph $H_{23}=P_{2}\times P_{3}$.}
\label{h23}
\end{figure}

\begin{table}[htbp]
\centering
\caption{A computation of the interior polynomial.}\label{tab:rec}
\begin{tabular}{lllllllll}
& &  & & & $(-1)^{|J|-1}I_{G - J}$ \\
\hline 
\\

\vspace{10pt}

&\hspace{40pt}\includegraphics[width=1cm]{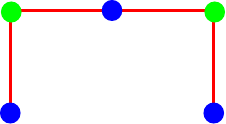}
&                                       &
&         
&\raisebox{+8pt}{\hspace{3pt}$1$}   
&\raisebox{+8pt}{$\times 1$}\\

\vspace{10pt}

& \includegraphics[width=1cm]{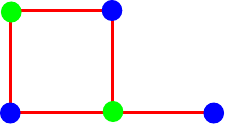}
& \includegraphics[width=1cm]{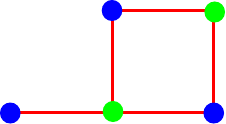}& 
&&\raisebox{+8pt}{$(1+1x)$}
& \raisebox{+8pt}{$\times 2$}\\

\vspace{10pt}

&\hspace{40pt}\includegraphics[width=1cm]{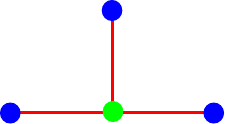}
&                                       &
&\raisebox{+8pt}{$-$\hspace{-10pt}}&\raisebox{+8pt}{$1$}
&\raisebox{+8pt}{$\times 1$}\\

 \vspace{10pt}

& \includegraphics[width=1cm]{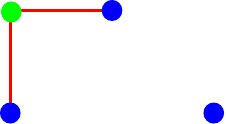}
& \includegraphics[width=1cm]{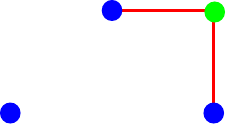}& 
& \raisebox{+8pt}{$-$\hspace{-10pt}}
&\raisebox{+8pt}{$(1-x)$}   
           & \raisebox{+8pt}{$\times 2$}\\

 \vspace{10pt}

&\hspace{40pt} \includegraphics[width=1cm]{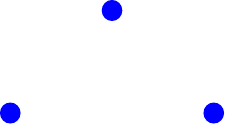}
& & 
& &\raisebox{+8pt}{$1-2x+x^2$}
           & \raisebox{+8pt}{$\times 1$}\\

\hline \\
 & &\hspace{-50pt}$\displaystyle I_{H_{23}}(x)$&$=$& &\hspace{3pt}$1+2x+x^2$
\end{tabular}
\end{table}

\end{example}

Next, we compute the interior polynomial of a complete bipartite graph. Let $K_{m,n}$ be a complete bipartite graph with color classes of sizes $m$ and $n$. We assume $m \le n$ without loss of generality, as the interior polynomial does not depend on the choice of color classes in the bipartite graph. 

The bipartite graph $K_{m,n} = (V \cup W, E)$ has $|V| = m$ and $|W| = n$. Since $m \le n$, the neighborhood of $W$ is exactly $V$, meaning that $|N_{K_{m,n}}(W)| = m$. Thus, we identify $W$ as a non-expanding set. Applying Corollary \ref{rec}, we obtain  
\[
I_G(x) = \sum_{\emptyset \ne J \subseteq W} (-1)^{|J| - 1} I_{K_{m,n} - J}(x).
\]
For any subset $ J \subseteq W $, deleting all vertices in $ J $ from $ W $ together with their incident edges yields a graph that is isomorphic to $ K_{m, n - |J|} $. Setting $ k = |J| $, we rewrite the formula as
\begin{equation}\label{combiprec}
I_{K_{m,n}}(x) = \sum_{k=1}^{n} (-1)^{k-1} \binom{n}{k} I_{K_{m,n-k}}(x).
\end{equation}
This result was previously obtained in \cite{GJ,K}. However, it also follows naturally from our approach.

\begin{theorem}\label{thm:complete}
Let $K_{m,n}$ be the complete bipartite graph. Then the interior polynomial of $K_{m,n}$ is given by
\[
I_{K_{m,n}}(x)=\sum_{j=0}^{\min\{m-1,n-1\}}\binom{m-1}{j}\binom{n-1}{j}x^j.
\]
\end{theorem}

\begin{proof}
First, we may assume $m \le n$ without loss of generality, since the interior polynomial does not depend on the choice of the color classes in the bipartite graph. We use induction on $(m,n)$ ordered lexicographically.

We first establish the base cases.  
For the degenerate case $K_{m,0}$ (i.e., one part is empty), the graph consists of $m$ isolated vertices, so we have
\[
I_{K_{m,0}}(x)=(1-x)^{m-1}.
\]
Using the generalized coefficient convention, where $\binom{-1}{j}=(-1)^j$, this formula can be rewritten as
\[
\sum_{j=0}^{m-1}\binom{m-1}{j}\binom{-1}{j}x^j=\sum_{j=0}^{m-1}\binom{m-1}{j}(-1)^j x^j=(1-x)^{m-1}.
\]
Thus, the formula holds for $K_{m,0}$.

Next, consider the case $m=1$. The graph $K_{1,n}$ is a tree, so we have
\[
I_{K_{1,n}}(x)=\sum_{j=0}^{0}\binom{0}{j}\binom{n-1}{j}x^j=1.
\]

Now, assume that for all pairs $(m',n')$ such that either $m'<m$ or $m'=m$ and $n'<n$, the interior polynomial satisfies
\[
I_{K_{m',n'}}(x)=\sum_{j=0}^{m'-1}\binom{m'-1}{j}\binom{n'-1}{j}x^j.
\]
We proceed to prove the formula for $K_{m,n}$ using the recurrence relation \eqref{combiprec}. Applying the induction hypothesis, we derive
\begin{eqnarray*}
I_{K_{m,n}}(x) &=& \sum_{k=1}^{n} (-1)^{k-1} \binom{n}{k} I_{K_{m,n-k}}(x)\\
&=&\sum_{k=1}^{n}(-1)^{k-1}\binom{n}{k}\left(\sum_{j=0}^{m-1}\binom{m-1}{j}\binom{n-k-1}{j}x^j\right)\\
&=&\sum_{j=0}^{m-1}\binom{m-1}{j}x^j\left(\sum_{k=1}^{n}(-1)^{k-1}\binom{n}{k}\binom{n-k-1}{j}\right).
\end{eqnarray*}

To complete the proof, we need to establish the following identity involving binomial coefficients. For any integer $0 \le j \le n-1$, we have
\begin{equation}\label{eq:binom_identity}
\sum_{k=1}^{n}(-1)^{k-1}\binom{n}{k}\binom{n-k-1}{j}=\binom{n-1}{j}.
\end{equation}

We prove the identity by induction on $n$. First, consider the base case $n=1$. The only possible value of $j$ is $j=0$. In this case, both sides evaluate to 1. Thus, the identity holds for $n=1$.

Now, suppose that for a fixed $n \geq 1$ and for all $0 \le j \le n-1$, the equality
\[
\sum_{k=1}^{n}(-1)^{k-1}\binom{n}{k}\binom{n-k-1}{j} = \binom{n-1}{j}
\]
holds.

For $ n+1$, let $S_{(n+1,j)}$ denote the left-hand side. Then, we compute
\begin{eqnarray*}
S_{(n+1,j)} &=& \sum_{k=1}^{n+1} (-1)^{k-1}\binom{n+1}{k}\binom{n+1-k-1}{j} \\
&=& \sum_{k=1}^{n+1} (-1)^{k-1}\binom{n+1}{k}\binom{n-k}{j}  \\
&=& \sum_{k=1}^{n} (-1)^{k-1}\binom{n+1}{k}\binom{n-k}{j} + (-1)^{n+j}.
\end{eqnarray*}
Using the identity $\binom{n+1}{k} = \binom{n}{k} + \binom{n}{k-1}$, we obtain
\[
S_{(n+1,j)} = \sum_{k=1}^{n} (-1)^{k-1} \left( \binom{n}{k} + \binom{n}{k-1} \right) \binom{n-k}{j} + (-1)^{n+j}.
\]
Splitting the sum, we rewrite it as
\[
S_{(n+1,j)} = S' + S'' + (-1)^{n+j},
\]
where
\[
S' = \sum_{k=1}^{n} (-1)^{k-1} \binom{n}{k} \binom{n-k}{j}, \quad
S'' = \sum_{k=1}^{n} (-1)^{k-1} \binom{n}{k-1} \binom{n-k}{j}.
\]
For $S'$, using the identity $\binom{n-k}{j} = \binom{n-k-1}{j} + \binom{n-k-1}{j-1}$, we obtain
\[
S' = \sum_{k=1}^{n} (-1)^{k-1} \binom{n}{k} \binom{n-k-1}{j}
+ \sum_{k=1}^{n} (-1)^{k-1} \binom{n}{k} \binom{n-k-1}{j-1}.
\]
By the induction hypothesis, we have
\[
S' = \binom{n-1}{j} + \binom{n-1}{j-1} = \binom{n}{j}.
\]
For $S''$, changing variables by setting $l=k-1$, so that $l$ runs from $0$ to $n-1$, we obtain
\begin{eqnarray*}
S'' 
&=& \sum_{l=0}^{n-1} (-1)^{l}\binom{n}{l}\binom{n-l-1}{j} \\
&=&(-1)^{0}\binom{n}{0}\binom{n-0-1}{j}+\sum_{l=1}^{n} (-1)^{l}\binom{n}{l}\binom{n-l-1}{j}-(-1)^{n}\binom{n}{n}\binom{n-n-1}{j} \\
&=&\binom{n-1}{j}-\sum_{l=1}^{n} (-1)^{l-1}\binom{n}{l}\binom{n-l-1}{j}-(-1)^n\binom{-1}{j}
\end{eqnarray*}
By the induction hypothesis, we have
\[
S''=\binom{n-1}{j}-\binom{n-1}{j}-(-1)^{n+j}=-(-1)^{n+j}
\]
Combining all terms, we obtain
\[
S_{(n+1, j)} =S' + S'' +(-1)^{n+j} 
= \binom{n}{j}-(-1)^{n+j} +(-1)^{n+j} 
= \binom{n}{j}.
\]
Thus, we have established the identity \eqref{eq:binom_identity}, completing the proof.

From \eqref{eq:binom_identity}, we obtain 
\[
I_{K_{m,n}}(x)=\sum_{j=0}^{m-1}\binom{m-1}{j}\binom{n-1}{j}x^j.
\]
Using induction on $(m,n)$ ordered lexicographically, we conclude that the stated formula for the interior polynomial of $K_{m,n}$ holds for all $ m \le n$.
\end{proof}


\begin{thebibliography}{99}	
\bibitem{CCD}M.\ Beck and S.\ Robins.\ Computing the Continuous Discretely, New York, Springer. 2015.
\bibitem{homfly}P.\ Freyd, D.\ Yetter, J.\ Hoste, W.\ B.\ R.\ Lickorish, K.\ Millett and A.\ Ocneanu.\ A new polynomial invariant of knots and links, Bull.\ Amer.\ Math.\ Soc.\ 12, 1985, 239--246.
\bibitem{GJ} X.\ Guan and X.\ Jin, On coefficients of the interior and exterior polynomials, arXiv:2201.12531, 2022.
\bibitem{Hall}P. Hall. On representatives of subsets, J. London Math. Soc. 10, 1935, 26--30.
\bibitem{K}T. K\'alm\'an.\ A version of Tutte's polynomial for hypergraphs, Adv.\ Math.\ 244, 2013, 823--873.
\bibitem{KM}T.\ K\'alm\'an and H. Murakami.\ Root polytopes, parking functions, and the HOMFLY polynomial, Quantum Topology.\ 8(2), 2017, 205--248.
\bibitem{KP}T.\ K\'alm\'an and A.\ Postnikov.\ Root polytopes, Tutte polynomials, and a duality theorem for bipartite graphs, Proc.\ London Math.\ Soc.\ 114(3), 2017, 561--588.
\bibitem{kato}K.\ Kato.\ Interior polynomial for signed bipartite graphs and the HOMFLY polynomial, Journal of Knot Theory and its Ramifications Vol. 29 No. 12 2050077(2020).
\bibitem{LP}N.\ Li and A. Postnikov.\ Slicing zonotopes, unpublished, 2013.
\bibitem{L}W.\ B.\ R.\ Lickorish.\ Polynomial invariant for links, Bull.\ London Math.\ Soc.\ 20, 1988, 558--588.
\bibitem{morton}H.\ R.\ Morton.\ Seifert circles and knot polynomial, Math.\ Proc.\ Camb.\ Phil.\ Soc.\ 99, 1986, 107--109.
\bibitem{OH}H.\ Ohsugi and T.\ Hibi.\ Normal Polytopes Arising from Finite Graphs. J.\ Algebra.\ 207(1998), 409–-426.
\bibitem{P}A. Postnikov.\ Permutohedra, Associahedra, and Beyond, Int. Math. Res. Not. 2009, no.\ 6, 1026–-1106.
\end{thebibliography}
\end{document}